\documentclass[12pt]{article}

\usepackage{setspace}
\usepackage{graphicx} 

\usepackage{chngcntr}
\counterwithin{figure}{section}
\usepackage[T1]{fontenc}
\usepackage{amsmath}
\usepackage{amsthm}
\usepackage{multirow}
\usepackage{amssymb}
\usepackage{mathrsfs} 
\usepackage{enumitem}
\usepackage{comment}
\usepackage{mathtools}
\usepackage{hyperref}
\usepackage{tikz-cd}
\usetikzlibrary{matrix,arrows,decorations.pathmorphing}

\usepackage{authblk}

\usepackage{bibunits}

\newtheorem{definition}{Definition}[section]
\newtheorem{remark}[definition]{Remark}

      \newtheorem{proposition}[definition]{Proposition}
            
      \newtheorem{proposition-definition}[definition]{Proposition-Definition}
      \newtheorem{definition-proposition}[definition]{Definition-Proposition}
      \newtheorem{theorem}[definition]{Theorem}
\newtheorem{example}{Example}
      \newtheorem{lemma}[definition]{Lemma}
      \newtheorem{corollary}[definition]{Corollary}

\numberwithin{equation}{section}

\def\End{{\mathrm {End}}}

\def\id{{\text{id}}}

\def\max{{\mathrm {max}}}

\def\Dom{{\mathrm {Dom}}}

\def\pb{\bar{\partial}}
\def\db{\bar{\delta}}
\def\nb{\bar{\nabla}_\theta}
\def\O{\Omega^{0,1}(S_q^2)}
\def\nb{\overline{\nabla}_\theta}

\def\Z{{\mathbb Z}}
\def\N{{\mathbb N}}
\def\C{{\mathbb C}}

\def\A{\mathcal{A}}

\def\I{\bar{I}}

\def\L{\mathcal{L}}

\def\d{\bar{\delta}}

\newcommand{\bb}[1]{\begin{bmatrix}#1\end{bmatrix}}

\def\arcpt{${{\lower3pt\hbox{$^{\prime\prime}$}}\atop{\raise4pt\hbox{.}}}$}

\usepackage{graphicx}

\usepackage[english]{babel}
\usepackage[utf8]{inputenc}
\usepackage{fancyhdr}
 
\def\day{}
\usepackage{multicol}

 \usepackage{url}
\usepackage{hyperref}
\hypersetup{
    colorlinks=true,
    linkcolor=blue,
    filecolor=magenta,      
    urlcolor=cyan,
}
 
\begin{document}


\title{Non-standard Holomorphic Structures on Line Bundles over the Quantum Projective Line}
\author[1]{Mary Graveman}
\author[2]{Landen La Rue}
\author[3]{Lillian MacArthur}
\author[4]{Hunter Pesin}
\author[5]{Zhaoting Wei}

\affil[1]{Lawrence University, Appleton, WI 54911, USA, 
maryhelen.graveman@lawrence.edu}
\affil[2]{University of Wyoming,  Laramie, WY 82071, USA, landen5larue@gmail.com}
\affil[3]{Harvard University, Cambridge, MA 02138, USA, 
lcmacarthur@college.harvard.edu}
\affil[4]{Cornell University,
Ithaca, NY 14853, USA, hp354@cornell.edu}
\affil[5]{Department of Mathematics, East Texas A\&M University, Commerce, TX 75428, USA, zhaoting.wei@etamu.edu}

\date{}

\maketitle

\begin{abstract}
In this paper we study non-standard holomorphic structures on line bundles over the quantum projective line $\C P^1_q$. We show that there exist infinitely many non-gauge equivalent holomorphic structures on those line bundles. This gives a negative answer to a question raised by  Khalkhali, Landi, and Van Suijlekom in 2011.
\end{abstract}

\section{Introduction}

\hspace{5mm} Over the past three decades, noncommutative differential geometry has witnessed substantial progress \cite{connes1994noncommutative}. By contrast, the corresponding theory of noncommutative complex geometry is still at a relatively early stage of development. An important step in this direction was made by D'Andrea,  Dabrowski, and Landi in \cite{d2008noncommutative}, where the authors introduced, for a deformation parameter $0<q<1$, the quantum projective plane $\C P^2_q$. These spaces provide  rich and instructive examples of  noncommutative complex manifolds. For higher dimensional cases and further discussion, see \cite{d2010dirac}, \cite{d2010anti}, and \cite{d2013geometry}. Khalkhali, Landi, and Van Suijlekom in \cite{khalkhali2011holomorphic}  studied holomorphic structures on the quantum projective line $\C P^1_q$ in details.  They  demonstrated that many of the classical features of the complex projective line $\C P^1$ continue to hold in the quantum setting. In particular, for each $n\in\mathbb{Z}$ they constructed holomorphic line bundles $\L_n$ on $\C P^1_q$, which may be regarded as noncommutative analogs of the classical line bundles $\mathcal{O}(n)$ on $\C P^1$. 

A fundamental property of the classical line bundle $\mathcal{O}(n)$ over $\C P^1$ is that its holomorphic structure is unique up to gauge equivalence. More concretely, a holomorphic structure on a complex vector bundle $E$ over a complex manifold $X$ is given by a flat $\bar{\partial}$-connection
\[
\overline{\nabla}: E \longrightarrow \Omega^{0,1}(X)\otimes E,
\]
and two such structures $\overline{\nabla}_1$ and $\overline{\nabla}_2$ are said to be gauge equivalent if there exists an invertible bundle map $g\in \operatorname{Aut}(E)$ such that
\[
g\circ \overline{\nabla}_1\circ g^{-1} = \overline{\nabla}_2.
\]
It is a classical fact that any holomorphic structure on $\mathcal{O}(n)$ is gauge equivalent to the standard one. Motivated by this, the authors of \cite[Page 872]{khalkhali2011holomorphic} asked whether the same statement remains true for the quantum line bundles $\L_n$ on $\C P^1_q$.

The purpose of this paper is to provide a negative answer to this question.

\begin{theorem}[See Theorem \ref{thm: gauge equivalent to trivial criterion} below]
For $0<q<1$, each quantum line bundle $\L_n$ over $\C P^1_q$ admits a flat $\bar{\partial}$-connection $\overline{\nabla}$ which is not gauge equivalent to the standard $\bar{\partial}$-connection.
\end{theorem}

We refer to such $\bar{\partial}$-connections as \emph{non-standard holomorphic structures}. Our work further investigates the nature of these non-standard structures and their mutual gauge equivalence. A key observation is that the dimension of spaces of holomorphic sections of non-standard holomorphic structures can grow indefinitely.

\begin{proposition}[See Corollary \ref{coro: ker_dim_lower_bound} below]
For $0<q<1$, on $\L_0$ over $\C P^1_q$, and for every $N\in\mathbb{N}$, there exists a $\bar{\partial}$-connection $\overline{\nabla}$ such that
\[
\dim \ker \overline{\nabla} \geq N.
\]
\end{proposition}

The above theorem, combined with the general fact that $\dim \ker \overline{\nabla}$ is finite, implies the following theorem, which is the main result of this paper:

\begin{theorem}[See Theorem \ref{thm: infinitely many gauge class} below]
For $0<q<1$, the line bundle $\L_n$ over $\C P^1_q$ carries infinitely many holomorphic structures, no two of which are gauge equivalent.
\end{theorem}

\medskip
\noindent
\textbf{Organization of the paper.}  
In Section \ref{section: review of quanum line}, we review the construction of the quantum projective line and recall the definition of holomorphic structures in this setting. Section \ref{section: non standard holomorphic structures} focuses on a distinguished sub-$C^*$-algebra of the quantum projective line, where we explicitly construct non-standard holomorphic structures and analyze their gauge equivalence classes. In Section \ref{section: holomorphic sections}, we study the holomorphic sections associated with these structures and establish the existence of classes with arbitrarily large finite dimension. Finally in Section \ref{section: future work} we briefly discuss some possible future work.

\subsection*{Acknowledgments}
The authors would like to thank Jonathan Block for his encouragement and valuable suggestions. They are also grateful to David Gao and R\'{e}amonn \'{O} Buachalla for kindly answering questions related to this work,  to Ryszard Szwarc for directing them to the reference \cite{koekoek2010hypergeometric}.  They want to thank Francesco D'Andrea for helpful comments on this work and   Satyajit Guin for pointing out a mistake in the first version of this paper.

This work was carried out during the TADM-REU program in Summer 2025, supported by the National
Science Foundation under Grant No DMS-2243991. The authors would like to thank East Texas A\&M University for providing an excellent research environment throughout the program. They also want to thank  Padmapani Seneviratne and Mehmet Celik for their organization and support during the program.

Z.W.'s research is partially supported by the AMS-Simons Research Enhancement Grants for Primarily Undergraduate Institution (PUI) Faculty.

\section{Review of the Quantum Projective line and line bundles on it}\label{section: review of quanum line}
In this section we mainly follow \cite[Section 3.1]{khalkhali2011holomorphic}.

\subsection{The quantum projective line  $\C P^1_q$}\label{subsection: quantum projective line}
The algebra $\A(SU_q(2))$ is the unital Hopf $*$-algebra with defining matrix \[
\bb{a & -qc^* \\ c & a^*}
\]
that is, the unital Hopf $*$-algebra defined by the relations 
\begin{equation}\label{eq: relations in A(SUq(2))}
    \begin{split}
    &ac=qca,~ac^*=qc^*a,~cc^*=c^*c,\\
    & a^*a+c^*c=aa^*+q^2cc^*=1.
    \end{split}
\end{equation}

It can be shown that $\A(SU_q(2))$ is a compact quantum group.


The algebra $\A(S_q^2)\to B(\ell^2)$ is a $*$-subalgebra of $\A(SU_q(2))$ which is given by the generators 
\begin{equation}\label{eq: generators of A(CP1)}
B_- =ac^*,~ B_+=ca^*,~ B_0=cc^*,
\end{equation}
one can calculate that these generators must obey the basic relations

\begin{equation}
    \begin{split}
        &B_-B_0=q^2B_0B_{-},~ B_{+}B_0=q^{-2}B_0B_{+},\\
        &B_-B_+=q^2B_0(1-q^2B_0),~B_+B_-=B_0(1-B_0)\\
        &B_0^*=B_0,~ B_+^*=B_-.
    \end{split}
\end{equation}

\begin{remark}
In the classical case where $q = 1,$ and we consider $S^2$ as the standard sphere, we see that $B_-, B_+,$ and $B_0$ correspond to $\frac{x-iy}{2}, \frac{x+iy}{2},$ and $\frac{1-z}{2}$ in $\A(S^2).$
\end{remark}

Line bundles $\L_n$ can be defined on $\A(S_q^2)$ as the $\A(S_q^2)$-sub-bimodules of $\A(SU_q(2)).$ These are generated by 

\begin{equation}
    \begin{split}
        &\{(c^*)^m(a^*)^{n-m},~m=0,\ldots n\} \text{ for }n\geq 0;\\
        &\{c^ma^{-n-m},~m=0,\ldots -n\} \text{ for }n\leq 0
    \end{split}
\end{equation}
Note that $\L_n$'s are projective $\A(S_q^2)$-modules.

One can observe that 
\begin{equation}\label{eq: line bundle product}
\L_0=\A(S_q^2),~\L_n^*=\L_{-n},~ \L_m\otimes_{\A(S_q^2)}\L_n\cong \L_{m+n}.
\end{equation}
If we denote by $\End_{\A(S_q^2)}(\L_n)$ the ring endomoprhisms of $\L_n$ as left $\A(S_q^2)$-modules, then we have 
\begin{equation}\label{eq: end of line bundles}
\End_{\A(S_q^2)}(\L_n)=\A(S_q^2),
\end{equation}
where the right hand side means right multiplication by $\A(S_q^2)$.

In \cite[Equation (3.18)]{khalkhali2011holomorphic} the authors introduced the $1$-form
$$
\omega_-:=c^* da^*-qa^*dc^*
$$
which satisfies ( \cite[Equation (3.21)]{khalkhali2011holomorphic})
\begin{equation}\label{eq: omega- commuting}
\begin{split}
&\omega_-a=q^{-1}a\omega_-,~\omega_-c=q^{-1}c\omega_-,\\
&\omega_-a^*=qa^*\omega_-,~ \omega_-c^*=qc^*\omega_-
\end{split}
\end{equation}
It is clear from \eqref{eq: omega- commuting} that $\omega_-$ commutes with elements in $\A(S_q^2)$.

Let $\Omega^{0,1}(\C P^1_q)$ be $\L_{-2}\omega_-$ as an $\A(S_q^2)$-bimodule. We can define the $\pb$ operator on $\A(S_q^2)$ as the map generated by the actions:

\begin{equation}\label{eq: pb on generators}
    \overline{\partial}(B_0)=-q^{-\frac{1}{2}}ca\omega_-,~\overline{\partial}(B_+)=q^{\frac{1}{2}}c^2\omega_-,~\overline{\partial}(B_-)=-q^{-\frac{1}{2}}a^2\omega_-.
\end{equation}
This makes $(\A(S_q^2), \overline{\partial})$ an algebra with complex structure in the sense of \cite[Definition 2.1]{khalkhali2011holomorphic}. From now on we denote $\A(S_q^2)$ by $\A(\C P^1_q)$.

\begin{lemma}\label{lemma: B0 pb B-  B+commutes}
    We have the following relations
    \begin{equation}\label{eq: B0 pb B- B+ commutes}
    \begin{split}
        &(\pb B_0 )B_0=q^2 B_0 \pb B_0, ~ (\pb B_- )B_-=q^2 B_- \pb B_-, ~(\pb B_+ )B_+=q^2 B_+ \pb B_+,\\
        & (\pb B_0 )B_-=B_-\pb B_0=q^2 B_0 \pb B_-,~(\pb B_0) B_+ = -\pb B_+ + q^2 B_0 \pb B_+.
        \end{split}
    \end{equation}
\end{lemma}
\begin{proof} They are direct consequences of \eqref{eq: relations in A(SUq(2))}, \eqref{eq: generators of A(CP1)}, and \eqref{eq: pb on generators}.
\end{proof}

\begin{lemma}\label{lemma: twist flip map}[\cite{khalkhali2011holomorphic} Lemma 3.6]
For any integer n, there is a twisted flip isomorphism
\begin{equation}\label{eq: twist flip isomorphism}
\Phi_{(n)}: \L_n\otimes_{\A(\C P^1_q)} \Omega^{0,1}(\C P^1_q)\overset{\sim}{\to}\Omega^{0,1}(\C P^1_q) \otimes_{\A(\C P^1_q)} \L_n
\end{equation}
as $\A(\C P^1_q)$-bimodules.
\end{lemma}

\begin{lemma}\label{lemma: associativity of Phi}
    Under the isomorphism $\L_m\otimes_{\A(\C P^1_q)}\L_n\cong \L_{m+n}$ in \eqref{eq: line bundle product}, we have the identity
    \begin{equation}\label{eq: associatitity of Phi}
   \Phi_{(m+n)} =(\Phi_{(m)}\otimes \id)\circ(\id\otimes \Phi_{(n)})
    \end{equation}
    as isomorphisms from $\L_m\otimes_{\A(\C P^1_q)}\L_n\otimes_{\A(\C P^1_q)} \Omega^{0,1}(\C P^1_q)$ to $\Omega^{0,1}(\C P^1_q) \otimes_{\A(\C P^1_q)} \L_m\otimes_{\A(\C P^1_q)}\L_n$.
\end{lemma}
\begin{proof}
    Since $\Omega^{0,1}(\C P^1_q)= \L_{-2}\omega_-$, \eqref{eq: associatitity of Phi} follows from \eqref{eq: omega- commuting} and the associativity of tensor products.
\end{proof}

\subsection{$\pb$-connections on line bundles}\label{subsection: standard holomorphic structures on line bundles}
\begin{definition}\label{def: left module connection}
    Let $\mathcal{E}$ be a left $\A(\C P^1_q)$-module. A left $\pb$-connection on $\mathcal{E}$ is a linear map $\overline{\nabla}: \mathcal{E}\to \Omega^{0,1}(\C P^1_q)\otimes_{\A(\C P^1_q)}\mathcal{E}$ satisfying the left Leibniz rule
    $$
    \overline{\nabla}(fe)=\pb(f)\otimes_{\A(\C P^1_q)}e+f\overline{\nabla}(e) \text{ for any } f\in \A(\C P^1_q), ~e\in \mathcal{E}.
    $$

    We can define right $\pb$-connections on right $\A(\C P^1_q)$-modules in the same way.
\end{definition}

We will need following definition in later construction:

\begin{definition}\label{def: bimodule connection}[\cite{khalkhali2011holomorphic} Definition 2.11]
    Let $\mathcal{E}$ be  an $\A(\C P^1_q)$-bimodule. A left $\pb$-connection $\overline{\nabla}$  on $\mathcal{E}$ is called a bimodule  $\pb$-connection if there exists an $\A(\C P^1_q)$-bimodule isomorphism $$\sigma(\overline{\nabla}): \mathcal{E}\otimes_{\A(\C P^1_q)}\Omega^{0,1}(\C P^1_q)\overset{\sim}{\to}\Omega^{0,1}(\C P^1_q)\otimes_{\A(\C P^1_q)}\mathcal{E}$$ such that for any $f\in \A(\C P^1_q)$ and $s\in\mathcal{E}$, the following twisted right Leibniz rule holds
    \begin{equation}
        \overline{\nabla}(sf)=\overline{\nabla}(s)f+\sigma(\overline{\nabla})(s\pb(f)).
    \end{equation}
\end{definition}

\begin{definition}\label{def: left holomorphic structure}
    Let $\mathcal{E}$ be a left $\A(\C P^1_q)$-module. A left holomorphic structure on $\mathcal{E}$ is a flat left $\pb$-connection on $\mathcal{E}$, i.e. a left  $\pb$-connection $\overline{\nabla}$ on $\mathcal{E}$ such that $\overline{\nabla}\circ \overline{\nabla}=0$.

    We can define right holomorphic structures on right $\A(\C P^1_q)$-modules in the same way.
\end{definition}

\begin{remark}
    For $\A(\C P^1_q)$-modules, the condition $\overline{\nabla}\circ \overline{\nabla}=0$ is automatically satisfied by dimension reason.
\end{remark}

We can define the standard $\pb$-connection $\overline{\nabla}^{(n)}: \L_n \to \Omega^{0,1}(\C P^1_q)\otimes_{\mathcal{A}(\C P^1_q)}\L_n$. In particular on $\L_0$ the $\pb$-connection $\overline{\nabla}^{(0)}$ coincides with $\pb$.

Observe that for $f\in  \A(\C P^1_q)$ and $s\in \L_n,$ we satisfy the Leibniz Rule \[
\overline{\nabla}^{(n)}(fs)=\pb(f)s+f\overline{\nabla}^{(n)}(s).
\]

According to \cite[Proposition 3.7]{khalkhali2011holomorphic}, the standard $\pb$-connection $\overline{\nabla}^{(n)}$ also satisfies the Leibniz rule with respect to the right multiplication
\begin{equation}\label{eq: right Leibniz rule}
\overline{\nabla}^{(n)}(sf)=\overline{\nabla}^{(n)}(s)f+\Phi_{(n)}(s \otimes\pb(f)),
\end{equation}
where $\Phi_{(n)}$ is the twist flip isomorphism in \eqref{eq: twist flip isomorphism}. In other words on each $\L_n$ the standard $\pb$-connection $\overline{\nabla}^{(n)}$ is a bimodule $\pb$-connection in the sense of Definition \ref{def: bimodule connection}.

The standard $\pb$-connection induces the following cochain complex: \[
0 \xrightarrow{0} \L_n \xrightarrow{\overline{\nabla}^{(n)}} \O \otimes \L_n \xrightarrow{0} 0
\]

\begin{proposition}\label{prop: 0,0-th cohomology for Ln}[\cite{khalkhali2011holomorphic} Theorem 4.4]
With the standard $\pb$-connection, the $(0,0)$-cohomologies on $\L_n$ is given by
\begin{equation}
H^{0,0}_{\overline{\nabla}^{(n)}}(\L_n)=\begin{cases}0, ~n>0\\ \C ^{|n|+1},~ n\leq 0\end{cases}
\end{equation}
\end{proposition}
In particular $H^{0,0}_{\overline{\nabla}^{(0)}}(\L_0)=\C $.

We also have the following result on $H^{0,1}$
\begin{proposition}\label{prop: 0,1th cohomology for L0}[\cite{d2013geometry} Proposition 7.2]
    With the standard $\pb$-connection, we have
\begin{equation}
H^{0,1}_{\overline{\nabla}^{(0)}}(\L_0)=0.
\end{equation}
\end{proposition}

Since the difference of any two $\pb$-connections is $\A(\C P^1_q)$-linear, any $\pb$-connection $\L_n$ is expressible as $\overline{\nabla}^{(n)} + D$ where $\overline{\nabla}^{(n)}$ is the standard $\pb$-connection on $\L_n$ and $$D \in Hom_{\A(\C P^1_q)}(\L_n, \Omega^{0,1}(\C P^1_q) \otimes_{\A(\C P^1_q)} \L_n).$$

Moreover by Lemma \ref{lemma: twist flip map} and \eqref{eq: line bundle product} such $D$ is realizable through right multiplication by a $(0,1)$ form, so we choose to express holomorphic structures of $\L_n$ on $\C P^1_q$ as
\begin{equation}\label{eq: pb-connections omega}
    \nb^{(n)}(s) := \overline{\nabla}^{(n)} s - \Phi_{(n)}(s\theta).
\end{equation}

\begin{remark}
    $\nb^{(n)}$  is a left $\pb$-connection on $\L_n$, but unlike the standard $\pb$-connection $\overline{\nabla}^{(n)}$, $\nb^{(n)}$ is not a bimodule $\pb$-connection on $\L_n$ in general.
\end{remark}

The following result generalized \cite[Proposition 3.8]{khalkhali2011holomorphic}.

\begin{definition-proposition}\label{def-prop: tensor product connection}
Let $\overline{\nabla}^{(m)}$ be the standard $\pb$-connection on $\L_m$ and $\nb^{(n)}$ be a left $\pb$-connection on $\L_n$, then we define the tensor product $\overline{\nabla}^{(m)}\otimes \nb^{(n)}$ as 
\begin{equation}
    \overline{\nabla}^{(m)}\otimes \nb^{(n)}:=\overline{\nabla}^{(m)}\otimes \id+ (\Phi_{(m)}\otimes \id)\circ (\id\otimes \nb^{(n)}).
\end{equation}
$\overline{\nabla}^{(m)}\otimes \nb^{(n)}$ is a left $\pb$-connection on $\L_{m+n}$. Moreover we have
\begin{equation}\label{eq: tensor product connections equal}
\overline{\nabla}^{(m)}\otimes \nb^{(n)}=\nb^{(m+n)}.
\end{equation}
\end{definition-proposition}
\begin{proof}
    First we check $\overline{\nabla}^{(m)}\otimes \nb^{(n)}$ is well defined. For $s\in \L_m$, $t\in \L_n$ and $f\in \A(\C P^1_q)$, we have 
    $$
    (\overline{\nabla}^{(m)}\otimes \nb^{(n)})(sf\otimes t)=\overline{\nabla}^{(m)}(sf)\otimes t+(\Phi_{(m)}\otimes \id) (sf \otimes \nb^{(n)}(t)).
    $$

    By \eqref{eq: right Leibniz rule} we know 
    $\overline{\nabla}^{(m)}(sf)=\overline{\nabla}^{(m)}(s)f+\Phi_{(m)}(s\otimes \pb(f))$
    hence
      \begin{equation*}
      \begin{split}
    (\overline{\nabla}^{(m)}\otimes \nb^{(n)})(sf\otimes t)=&\overline{\nabla}^{(m)}(s)f\otimes t+\Phi_{(m)}(s\otimes \pb(f))\otimes t+(\Phi_{(m)}\otimes \id) (sf \otimes \nb^{(n)}(t))\\
    =&\overline{\nabla}^{(m)}(s)\otimes ft+(\Phi_{(m)}\otimes \id) (s\otimes \pb(f)\otimes t+s \otimes f\nb^{(n)}(t)).
    \end{split}
    \end{equation*}
    On the other hand
      \begin{equation*}
      \begin{split}
    (\overline{\nabla}^{(m)}\otimes \nb^{(n)})(s\otimes ft)=&\overline{\nabla}^{(m)}(s)\otimes ft+ (\Phi_{(m)}\otimes \id)(s\otimes \nb^{(n)}(ft))\\
    =&\overline{\nabla}^{(m)}(s)\otimes ft+(\Phi_{(m)}\otimes \id)(s\otimes \pb(f)\otimes t+s\otimes f\nb^{(n)}t).
     \end{split}
    \end{equation*}
Therefore $(\overline{\nabla}^{(m)}\otimes \nb^{(n)})(sf\otimes t)=(\overline{\nabla}^{(m)}\otimes \nb^{(n)})(s\otimes ft)$.

Next we check that $\overline{\nabla}^{(m)}\otimes \nb^{(n)}$ is a left $\pb$-connection. For $s\in \L_m$, $t\in \L_n$ and $f\in \A(\C P^1_q)$, we have
  \begin{equation*}
      \begin{split}
    (\overline{\nabla}^{(m)}\otimes \nb^{(n)})(fs\otimes t)=&\overline{\nabla}^{(m)}(fs)\otimes t+(\Phi_{(m)}\otimes \id) (fs \otimes \nb^{(n)}(t))\\
    =& \pb(f)s\otimes t+f\overline{\nabla}^{(m)}(s)\otimes t+(\Phi_{(m)}\otimes \id) (fs \otimes \nb^{(n)}(t)).
     \end{split}
         \end{equation*}
     Since $\Phi_{(m)}$ is an $\A(\C P^1_q)$-bimodule map, we have
     $$(\Phi_{(m)}\otimes \id) (fs \otimes \nb^{(n)}(t))=f(\Phi_{(m)}\otimes \id) (s \otimes \nb^{(n)}(t))$$ hence
$$(\overline{\nabla}^{(m)}\otimes \nb^{(n)})(fs\otimes t)=\pb(f)s\otimes t+f(\overline{\nabla}^{(m)}\otimes \nb^{(n)})(s\otimes t).$$

Lastly we check \eqref{eq: tensor product connections equal}. For  $s\in \L_m$ and $t\in \L_n$ we have
 \begin{equation}\label{eq: tensor connection expansion}
      \begin{split}(\overline{\nabla}^{(m)}\otimes \nb^{(n)})(s\otimes t)=&\overline{\nabla}^{(m)}(s)\otimes t+(\Phi_{(m)}\otimes \id) (s \otimes \nb^{(n)}(t))\\ 
      =&\overline{\nabla}^{(m)}(s)\otimes t+(\Phi_{(m)}\otimes \id) (s \otimes \overline{\nabla}^{(n)}(t)-s\otimes\Phi_{(n)}(t\otimes \theta))\\
      =&\overline{\nabla}^{(m)}(s)\otimes t+(\Phi_{(m)}\otimes \id) (s \otimes \overline{\nabla}^{(n)}(t))-(\Phi_{(m)}\otimes \id)(s\otimes\Phi_{(n)}(t\otimes \theta)).\end{split}
         \end{equation}
By \cite[Proposition 3.8]{khalkhali2011holomorphic} we know that $\overline{\nabla}^{(m)}\otimes \id+ (\Phi_{(m)}\otimes \id)\circ (\id\otimes \overline{\nabla}^{(n)})=\overline{\nabla}^{(m+n)}$. The equality in \eqref{eq: tensor product connections equal} then follows from Lemma \ref{lemma: associativity of Phi} and \eqref{eq: tensor connection expansion}.
\end{proof}

\begin{remark}
$\overline{\nabla}^{(m)}\otimes \nb^{(n)}$ is  flat by dimension reason.
\end{remark}

\begin{remark}
    In general we cannot define the tensor production connection $\overline{\nabla}_{\theta_1}^{(m)}\otimes \overline{\nabla}_{\theta_2}^{(n)}$ of two left $\pb$-connections $\overline{\nabla}_{\theta_1}^{(m)}$ and $\overline{\nabla}_{\theta_2}^{(n)}$.
\end{remark}

\begin{remark}
In the sequel we will usually omit the superscript "$(n)$" in the notation of $\pb$-connections and simply denote it by $\nb$.
\end{remark}

\subsection{$C^*$-completions, $L^2$-completions, and the spectral triple}
As in \cite[Section 3.2]{khalkhali2011holomorphic}, we denote by $\mathcal{C}(SU_q(2))$ the $C^*$-completion of $\mathcal{A}(SU_q(2))$. By definition $\mathcal{C}(SU_q(2))$ is the universal $C^*$-algebra generated by $a$ and $c$ subject to relations in \eqref{eq: relations in A(SUq(2))}.
Moreover, the exists a unique left invariant Haar state $h$ on $\mathcal{C}(SU_q(2))$ such that $h(1)=1$. As a result we can consider  $L^2(SU_q(2))$ via the GNS-construction on $\mathcal{C}(SU_q(2))$.

We can also define the $C^*$-subalgebra of $\mathcal{C}(SU_q(2))$ generated by $B_0$, $B_+$, and $B_-$, which we denote by $\mathcal{C}(\C P^1_q)$; and its $L^2$-completion $L^2(\C P^1_q)$. We consider $\mathcal{C}(\C P^1_q)$ and $L^2(\C P^1_q)$ the algebras of continuous and $L^2$-functions on $\C P^1_q$, 
respectively.

We shall introduce a $C^*$-representation of $\mathcal{C}(\C P^1_q)$. Let $\ell^2$ be the standard separable Hilbert space with orthonormal basis $\{e_n\}_{n\geq 0}$, and $B(\ell^2)$ be the $C^*$-algebra of bounded operators on $\ell^2$.

\begin{proposition}\label{prop: standard representation}[\cite{podles1987quantum} Proposition 4, \cite{aguilar2018podles} Proposition 4.1]
There exists a faithful representation $\pi: \mathcal{C}(\C P^1_q)\to B(\ell^2)$ such that
  \begin{equation}\label{eq: standar representation}
    \begin{split}
        \pi(B_-)(e_n)&=q^{n}\sqrt{1-q^{2n}}e_{n-1};\\
        \pi(B_0)(e_n)&=q^{2n}e_n;\\
        \pi(B_+)(e_n)&=q^{n+1}\sqrt{1-q^{2n+2}}e_{n+1}.
    \end{split}
    \end{equation}

    In particular, the spectrum of $B_0$  is $\{0\} \cup \{q^{2n} | n \in \Z_{\geq 0} \}$.
\end{proposition}

\begin{remark}
We can further show that $\mathcal{C}(\C P^1_q)$ is $*$-isomorphic to the $C^*$-subalgebra of $B(\ell^2)$ generated by $1$ and all compact operators. Nevertheless we do not need this fact in our paper.
\end{remark} 

As in \cite[Section 3.3]{khalkhali2011holomorphic}, we can define $\Gamma(\L_n)$  and $L^2(\L_n)$ as spaces of continuous and $L^2$-sections of the line bundle $\L_n$, respectively. It is clear that  $\Gamma(\L_n)$ is a $\mathcal{C}(\C P^1_q)$-bimodule. Similar to \eqref{eq: line bundle product} we have
\begin{equation}
\End_{\mathcal{C}(\C P^1_q)}(\Gamma(\L_n))=\mathcal{C}(\C P^1_q).
\end{equation}

In particular we can consider the Hilbert space $L^2(\Omega^{0,1}(\C P^1_q))$. According to \cite[Section 7.1]{d2013geometry}, the map
$\pb: \A(\C P^1_q)\to \Omega^{0,1}(\C P^1_q)$ has a Hermitian conjugate
\begin{equation}
\pb^{\dagger}:  \Omega^{0,1}(\C P^1_q)\to \A(\C P^1_q).
\end{equation}
Moreover we have the following theorem:

\begin{theorem}\label{thm: spectral triple}[\cite{d2010dirac} Theorem 6.2, see also \cite{das2020dolbeault} Theorem 6.21]
    $(\A(\C P^1_q), L^2(\Omega^{0,\bullet}(\C P^1_q)), D_{\pb})$ forms a spectral triple in the sense of \cite{connes1995noncommutative}, where $D_{\pb}:=\pb+\pb^{\dagger}$.
\end{theorem}

\begin{remark}\label{rmk: db closed map}
In particular, the map $\pb: \A(\C P^1_q)\to \Omega^{0,1}(\C P^1_q)$ extends to a closed map
$$\pb: L^2(\C P^1_q)\to L^2(\Omega^{0,1}(\C P^1_q)).$$
\end{remark}

We will need the following result
\begin{proposition}\label{prop: kernel of pb in completion}[\cite[Corollary 4.3]{khalkhali2011holomorphic}]
    There are no nontrivial holomorphic functions in $\Dom(\pb)\cap \mathcal{C}(\C P^1_q)$, i.e. we have $\ker\pb\cap \mathcal{C}(\C P^1_q)=\C $.
\end{proposition}

For any $\pb$-connection $\overline{\nabla}_{\theta}$ on $\L_0$, we  can define its Hermitian conjugate $\overline{\nabla}_{\theta}^{\dagger}$ and the Dirac operator $D_{\overline{\nabla}_{\theta}}$ in the same way. We have the following corollary:

\begin{corollary}\label{coro: any connection gives a spectral triple}
 For any $\pb$-connection $\overline{\nabla}_{\theta}$ on $\L_0$, $(\A(\C P^1_q), L^2(\Omega^{0,\bullet}(\C P^1_q)), D_{\overline{\nabla}_{\theta}})$ forms a spectral triple. In particular $\ker \overline{\nabla}_{\theta}$ is finite dimensional for any $\pb$-connection $\overline{\nabla}_{\theta}$ on $\L_0$.
\end{corollary}
\begin{proof}
    Since $\theta\wedge(-)$ is a bounded operator, the difference between $D_{\pb}$ and $D_{\overline{\nabla}_{\theta}}$ is a bounded self-adjoint operator. As spectral triples are preserved by bounded perturbations, $(\A(\C P^1_q), L^2(\Omega^{0,\bullet}(\C P^1_q)), D_{\overline{\nabla}_{\theta}})$ is still a spectral triple.

    The finite dimensionality of $\ker \overline{\nabla}_{\theta}$ follows from the fact that $D_{\overline{\nabla}_{\theta}}$ has compact resolvent.
\end{proof}

\section{Nontrivial Gauge Equivalency Classes of Holomorphic Structures on line bundles}\label{section: non standard holomorphic structures}

\subsection{Generalities on gauge equivalences of holomorphic structures}\label{subsection: gauge equivalence}
Khalkhali et. al extend the notion of gauge equivalence in the noncommutative case in \cite[Definition 2.9]{khalkhali2011holomorphic}. Two $\pb$-connections $\overline{\nabla}_{\theta_1}, \overline{\nabla}_{\theta_2}$ on  $\L_n$ are said to be gauge equivalent if there exists an invertible element $g \in \End_{\A(\C P^1_q)}(\L_n)$ such that \[
    \overline{\nabla}_{\theta_1} = g^{-1} \circ \overline{\nabla}_{\theta_2}\circ g.
\]

\begin{lemma}\label{lemma: gauge equivalence expanded}
    Two $\pb$-connections $\overline{\nabla}_{\theta_1}, \overline{\nabla}_{\theta_2}$ on  $\L_n$ are  gauge equivalent if and only if there exists an invertible element $g\in \A(\C P^1_q)$ such that
    $$
    \theta_1=g\theta_2 g^{-1}-\pb(g)g^{-1}.
    $$

    In particular, a $\pb$-connection $\overline{\nabla}_{\theta}$ is gauge equivalent to the standard $\pb$-connection $\overline{\nabla}$ if and only if there exists an invertible element $g\in \A(\C P^1_q)$ such that $\pb(g)=g\theta$.
\end{lemma}
\begin{proof}
    By \eqref{eq: end of line bundles} we know that $\End_{\A(\C P^1_q)}(\L_n)=\A(\C P^1_q)$ where the right hand side means right multiplication by elements in $\A(\C P^1_q)$. The result then follows from  \eqref{eq: right Leibniz rule} and \eqref{eq: pb-connections omega}.
\end{proof}

However, the condition that $g\in \A(\C P^1_q)$ in Lemma \ref{lemma: gauge equivalence expanded} is too restrictive as shown in the following example.

\begin{example}\label{example: exp equation for B_0}
Let $\theta=\pb(B_0)$. Using \eqref{eq: B0 pb B- B+ commutes} and induction we can get
\begin{equation}\label{eq: pb of B_0^n}
    \pb(B_0^n)=\sum_{k=0}^{n-1}q^{2k}B_0^{n-1}\pb(B_0)=q^{n-1}[n]_{q}B_0^{n-1}\pb(B_0),
\end{equation}
where $[n]_{q}=\frac{q^{n}-q^{-n}}{q-q^{-1}}$ is the $q$-integer as in \cite[(3.1)]{khalkhali2011holomorphic}. We define $[n]_{q}!:=\prod_{k=1}^n [k]_{q}$ and $[0]_{q}!:=1$. Then we can check that 
\begin{equation}\label{eq: g for B_0}
g:=\sum_{n=0}^{\infty}\frac{B_0^n}{q^{\frac{n(n-1)}{2}}[n]_{q}!}
\end{equation}
satisfies $\pb(g)=g\pb(B_0)$. Notice that Propositioin \ref{prop: standard representation} gives us $\lVert B_0\rVert=1$ hence the series in \eqref{eq: g for B_0} convergies. It is clear that
$$
g^{-1}=\sum_{n=0}^{\infty}\frac{(-1)^n B_0^n}{q^{\frac{n(n-1)}{2}}[n]_{q}!}
$$
hence $g\in \mathcal{C}(\C P^1_q)$ is invertible and $g\in \Dom(\pb)$ and $g^{-1}\in \Dom(\pb)$ but $g\notin \A(\C P^1_q)$.
\end{example}

Inspired by Example \ref{example: exp equation for B_0} we have the following modified definition.

\begin{definition}\label{def: gauge equivalence}
     We call two $\pb$-connections $\overline{\nabla}_{\theta_1}, \overline{\nabla}_{\theta_2}$ on  $\L_n$   gauge equivalent if there exists an invertible element $g\in \mathcal{C}(\C P^1_q)^{\times}\cap \Dom(\pb)$ such that $g^{-1}\in \Dom(\pb)$ and
     $$
         \theta_1=g\theta_2 g^{-1}-\pb(g)g^{-1}.
    $$
     hence
      \begin{equation}\label{eq: gauge equivalence general}
        \pb(g)= g\theta_2 -\theta_1 g.
     \end{equation}

     In particular, a $\pb$-connection $\overline{\nabla}_{\theta}$ on  $\L_n$ is gauge equivalent to the standard $\pb$-connection if there exists an invertible element $g\in \mathcal{C}(\C P^1_q)^{\times}\cap \Dom(\pb)$ such that $g^{-1}\in \Dom(\pb)$ and
     \begin{equation}\label{eq: gauge equivalence to standard general}
        \pb(g)=g\theta.
     \end{equation}
\end{definition}

By Proposition \ref{prop: 0,1th cohomology for L0} we know that $H^{0,1}_{\overline{\nabla}}(\L_0)=0.$ As a result for any $\theta\in \Omega^{0,1}(\C P^1_q)$ there exists an $f\in \A(\C P^1_q)$ such that $\theta=\pb (f)$. Therefore a $\pb$-connection $\overline{\nabla}_{\theta}=\overline{\nabla}_{\pb(f)}$ on  $\L_n$ is gauge equivalent to the standard $\pb$-connection if there exists an invertible element $g\in \mathcal{C}(\C P^1_q)^{\times}\cap \Dom(\pb)$ such that $g^{-1}\in \Dom(\pb)$ and
     \begin{equation}\label{eq: exponential equation}
        \pb(g)=g\pb(f),
     \end{equation}
which is a noncommutative analogue of the exponential equation.

\begin{remark}
Although it seems plausible that $g\in \mathcal{C}(\C P^1_q)^{\times}\cap \Dom(\pb)$ implies $g^{-1}\in \Dom(\pb)$, we do not know whether it is true in general. So we have to add $g^{-1}\in \Dom(\pb)$ to the definition. The authors want to thank Satyajit Guin for pointing out this. 

Also see Corollary \ref{coro: inverse domain d} below for a proof in a special case.
\end{remark}

\begin{remark}
If the algebra was commutative, then $g=\exp(f)$ would give a solution to \eqref{eq: exponential equation}. Hence we call \eqref{eq: exponential equation} the noncommutative exponential equation.
\end{remark}

\begin{remark}
In \cite{polishchuk2006analogues} Polishchuk studied the analogue of \eqref{eq: exponential equation} on noncommutative two-tori.
\end{remark}

The following lemma plays a key role in the contruction of non-standard holomorphic structures:

\begin{lemma}\label{lemma: non zero solution non gauge equivalent}
    If there exists a non-zero non-invertible $h \in \mathcal{C}(\C P^1_q)\cap \Dom(\pb)$ such that $\pb (h) = h \pb (f)$, then there cannot exist an invertible $g$ such that  $g\in \mathcal{C}(\C P^1_q)^{\times}\cap \Dom(\pb)$, and  $g^{-1}\in \Dom(\pb)$, and $\pb (g) = g \pb (f)$.
\end{lemma}
\begin{proof}
    For the sake of contradiction, let such a $g$ exist. Since $g$ is invertible, we can write   $h = ag$ for $a = h g^{-1}$. Since both $h$ and $g^{-1}$ are in $\Dom(\pb)$, so is $a$.  We then  have
    $$
    \pb(h)=\pb(ag)=\pb(a)g+a\pb(g)=\pb(a)g+ag\pb(f)=\pb(a)g+h\pb(f).
    $$
    Since $\pb (h) = h \pb (f)$, we have \[
     \pb(a) g = 0.
    \]
    Since $g$ is invertible, we then have $\pb(a) = 0$. By Proposition \ref{prop: kernel of pb in completion}, this means that $a\in \C $ is a constant. However, this implies that either $h = 0$ or that $h$ is invertible, a contradiction.
\end{proof}

\begin{proposition}\label{prop: one-one corrspondence gauge equivalence class}
There is a one-to-one correspondence between sets of gauge equivalence classes of holomorphic structures on $\L_m$ and $\L_n$ for any $m$ and $n$.
\end{proposition}
\begin{proof}
By Proposition \ref{def-prop: tensor product connection}, there exists a  one-to-one correspondence between  holomorphic structures on $\L_m$ and $\L_n$. The compatibility with gauge equivalences follows from \eqref{eq: tensor product connections equal} and \eqref{eq: gauge equivalence general}.
\end{proof}

\subsection{The $C^*$-subalgebra $C^*(1, B_0)$}\label{subsection: C^*(1, B_0)}

Let $C^*(1, B_0)$ denote the unital $C^*$-subalgebra of $\mathcal{C}(\C P^1_q)$ generated by $B_0$.
Since $B_0$ is self-adjoint hence normal,  by continuous functional calculus we have a $*$-isomorphism
\begin{equation}\label{eq: functional calculus B0}
\Psi: C^*(1, B_0)\overset{\sim}{\to}C(sp(B_0)),
\end{equation}
where $C(sp(B_0))$ is the $C^*$-algebra of continuous functions on $sp(B_0)$ the spectrum of $B_0$. In particular $\Psi$ maps $B_0$ to the indentity function $f(x)=x$ on $p(B_0)$.
Recall Proposition \ref{prop: standard representation} tells us
\begin{equation}\label{eq: sp B0}
sp(B_0)=\{0\} \cup \{q^{2n} | n \in \Z_{\geq 0} \}.
\end{equation}
Note that since $sp(B_0)$ only has a single limit point at zero, continuity of a function $f$ on $sp(B_0)$ is equivalent to $\lim_{n \to \infty} f(q^{2n}) = f(0).$ Additionally, an element $f \in C(sp(B_0))$ is invertible iff it never vanishes on the spectrum, as this is the necessary and sufficient condition for $1/f$ being well-defined.
 
We want to study the restriction of $\pb$ to $C^*(1, B_0)$ in more details. First we introduce the following operator.

\begin{definition}\label{def: db}
 Let $\C[x]$ denote the algebra of polynomials.   We define the linear map $\db: \C[x] \to \C[x]$ as 
 \begin{equation}\label{eq: def of db}
 \db(f)(x):=\frac{f(x)-f(q^2x)}{x-q^2x}.
 \end{equation}
\end{definition}

\begin{remark}
The same formula as \eqref{eq: def of db} appeared in \cite[Section 1.15]{koekoek2010hypergeometric}. Nevertheless  analytic properties of $\db$ like Proposition \ref{prop: db is closable} below have not been covered in \cite{koekoek2010hypergeometric} .
\end{remark}

\begin{remark}
The operator $\db$ is not a derivation on $\C[x]$ in the usual sense. Actually we can show that $\db$ satisfies a twisted Leibniz rule as in \cite[Equation (1.15..5)]{koekoek2010hypergeometric}, but we do not need this fact in our paper.
\end{remark}

If we define the dilation operator $m_{c}$ for $c\in \mathbb{R}$ on $\C[x]$ by 
\begin{equation}\label{eq: mc}
    m_c(f)(x):=f(cx),
\end{equation}
then \eqref{def: db} can be rewritten as
\begin{equation}
    \db(f)(x)=\frac{f(x)-m_{q^2}(f)(x)}{x-q^2x}.
\end{equation}
We can consider $\C[x]$ as a subspace of $C(sp(B_0))$ by restricting $f(x)$ to $sp(B_0)$. If $0<c<1$, then we can also extend $m_c$ to a bounded  operator on $C(sp(B_0))$.

\begin{lemma}\label{lemma: db and pb correspond}
The map $\db$ corresponds to $\pb$ under the functional calculus isomorphism \eqref{eq: functional calculus B0}. In more details, for any $f\in \C[x]\subset C(sp(B_0))$, we have
\begin{equation}\label{eq: db and pb correspond}
\pb(\Psi^{-1}(f))=(\Psi^{-1}(\db f))\pb B_0.
\end{equation}
\end{lemma}
\begin{proof}
The definition \eqref{eq: def of db} gives \[
\db(x^n) = \frac{1 - q^{2n}}{1 - q^2} x^{n-1},
\]
and $\db(1) = 0$. On the other hand 
\eqref{eq: pb of B_0^n} gives \[
\pb(B_0^n) =\sum_{k=0}^{n-1}q^{2k}B_0^{n-1}\pb(B_0)= \frac{1 - q^{2n}}{1 - q^2} B_0^{n-1} \pb B_0.
\]
Since $\Psi(B_0)=x$, the lemma then follows by linearity.
\end{proof}
By the  Stone-Weierstrass Theorem, $\C[x]$ is a dense subset of $C(sp(B_0))$. However the operator $\db$ is not bounded, for example for $f_n(x)=(1-x)^n$ we have $\|f_n\|=1$ when we take the maximal norm as elements in $C(sp(B_0))$. On the other hand $\|\db(f_n)\|\geq |\db(f_n)(0)|=n$.

Therefore we cannot extend $\db$ to an operator on $C(sp(B_0))$. To further study the analytic properties of $\db$, we introduce the 
  $\I$ operator inspired by \cite[Equation (1.15.7)]{koekoek2010hypergeometric}.

\begin{definition}\label{def: Ibar operator}
  The linear map $\I: \C[x]\to \C[x]$  is defined  by 
  \begin{equation}\label{eq: def of Ibar}
\bar{I}(x^n)=\frac{1-q^2}{1-q^{2n+2}}x^{n+1}
\end{equation}
and $\bar{I}(0)=0$.
\end{definition}

\begin{lemma}\label{lemma: I db inverse}
For all $f \in \C[x]$ we have 
\begin{equation}\label{eq: I db inverse}
\db(\I(f)) = f \text{ and }\I(\db(f)) = f-f(0).
\end{equation}
\end{lemma}
\begin{proof} It By direct computation we can check that \eqref{eq: I db inverse} holds for any $f(x)=x^n$. The general case then follows by linearity.
\end{proof}

\begin{lemma}\label{lemma: expansion of I}
    Given a $f \in \C[x]$ we have  \begin{equation}\label{eq: expansion of I}
    \I(f)(x) = (1-q^2)x \sum_{n=0}^{\infty} q^{2n}(m_{q^{2n}}f)(x),
    \end{equation}
    where $m_{q^{2n}}$ is defined in \eqref{eq: mc}.
\end{lemma}
\begin{proof}
   For $f(x)=x^k$, \eqref{eq: def of Ibar} gives
   \begin{equation*}
   \begin{split}
   \I(f)(x)=&(1-q^2)x\frac{x^k}{1-q^{2k+2}}\\
   =&(1-q^2)x\sum_{n=0}^{\infty}q^{(2k+2)n}x^k\\
   =&(1-q^2)x\sum_{n=0}^{\infty} q^{2n}(q^{2n}x)^k\\
   =&(1-q^2)x\sum_{n=0}^{\infty}q^{2n}(m_{q^{2n}}f)(x).
   \end{split}
   \end{equation*}
   The general case then follows by linearity.
\end{proof}

\begin{lemma}\label{lemma: bound of I}
    Given a $f \in \C[x]$, we have $\|\I(f)\| \leq \|f\|.$
\end{lemma}
\begin{proof}
   For $f\in \C[x]$, by Lemma \ref{lemma: expansion of I} we have  \[
    \|\I(f)\| = \|(1-q^2) x \sum_{n=0}^\infty q^{2n}(n_{q^{2n}} f) \| \leq   (1-q^2) \| x \|\cdot  \|  \sum_{n=0}^\infty q^{2n}(m_{q^{2n}} f) \| 
    \]
    Since $sp(B_0)=\{0\} \cup \{q^{2n} | n \in \Z_{\geq 0} \}\subset [0,1]$, we have $\|x\| = 1$. Hence \[
    \|\I(f)\| \leq  (1-q^2)  \sum_{n=0}^\infty \|   q^{2n}(m_{q^{2n}} f) \| = (1-q^2)  \sum_{n=0}^\infty q^{2n} \| (m_{q^{2n}} f) \|
    \]
    Now, since $q^{2n}\leq 1$, the dilation $m_{q^{2n}}$ does not increase the norm of $f,$ we have \[
    \|\I(f)\| \leq  (1-q^2)  \sum_{n=0}^\infty q^{2n} \| f \| = \|f\| \cdot \Big( (1-q^2)  \sum_{n=0}^\infty q^{2n} \Big) = \|f\|.
    \] 
\end{proof}

\begin{lemma}\label{lemma: I extend to C(spB0)}
    $\I$ extends to a  bounded map $\I: C(sp(B_0)) \to C(sp(B_0))$. Moreover, \eqref{eq: expansion of I} holds for any $f\in C(sp(B_0))$.
\end{lemma}
\begin{proof}
    By the Stone-Weierstrass theorem, $\C[x]$ is dense in $C(sp(B_0))$. The result then follows from Lemma \ref{lemma: bound of I}.
\end{proof}

\begin{lemma}\label{lemma: I injective}
   The map $\I: C(sp(B_0)) \to C(sp(B_0))$ is injective.
\end{lemma}
\begin{proof}
Let $f\in C(sp(B_0))$ be in the kernel of $\I$. Since $sp(B_0)=\{0\} \cup \{q^{2n} | n \in \Z_{\geq 0} \}$, for any $k\geq 0$ we have $\I(f)(q^{2k})=0$. By \eqref{eq: expansion of I} we have
$$
\I(f)(q^{2k})=(1-q^2)q^{2k} \sum_{n=0}^{\infty} q^{2n}(m_{q^{2n}}f)(q^{2k})=(1-q^2)q^{2k} \sum_{n=0}^{\infty} q^{2n}f(q^{2n+2k}).
$$
Therefore $\I(f)(q^{2k})=0$ implies 
\begin{equation}\label{eq: sum is zero1}
\sum_{n=0}^{\infty} q^{2n}f(q^{2n+2k})=0.
\end{equation}

Notice that we also have $\I(f)(q^{2(k+1)})=0$ hence 
\begin{equation}\label{eq: sum is zero2}\sum_{n=0}^{\infty} q^{2n}f(q^{2n+2(k+1)})=\sum_{n=0}^{\infty} q^{2n}f(q^{2n+2+2k})=0.\end{equation}
Compare \eqref{eq: sum is zero1} and \eqref{eq: sum is zero2} we get
$$f(q^{2k})=0 \text{ for any }k\geq 0.
$$
Since $f$ is continuous, we also get
$$f(0)=\lim_{k\to \infty}f(q^{2k})=0.$$
Hence $f\equiv 0$.
\end{proof}

\begin{proposition}\label{prop: db is closable}
    $\db$ is a closable operator on $C(sp(B_0))$.
\end{proposition}
\begin{proof}
    Recall that $\d$ being closable means that for all $\{f_n\} \in \C[x]$ such that $f_n \to 0$ and $\d f_n \to g,$ for some $g \in C(sp(B_0)),$ then $g = 0.$ 
    
    Now, by Lemma \ref{lemma: I extend to C(spB0)} we know $\I$ is bounded hence $\d f_n \to g$ implies $\I \d f_n \to \I g$. By Lemma \ref{lemma: I db inverse}, $\I \d f_n=f_n - f_n(0)$ hence we have $f_n - f_n(0) \to \I g.$ However, since $f_n \to 0,$  we also have $f_n - f_n(0) \to 0,$ which means that $\I g = 0.$ The injectivity of $\I$ as in Lemma \ref{lemma: I injective} then implies  $g=0$.
\end{proof}

Proposition \ref{prop: db is closable} tells us that we can extend $\db$ to a closed operator on $C(sp(B_0))$.

\begin{remark}
    We can deduce that $\db$ is closable from the fact that $\pb$ is a closed operator and the relation \eqref{eq: db and pb correspond}. We give a direct proof here because the operator $\I$ which is introduced in the proof is important in the proof of Proposition \ref{prop: delta_bar_domain} below.
\end{remark}

\begin{proposition} \label{prop: delta_bar_domain}
    $f\in C(sp(B_0))$ is in the domain of $\d$ if and only if \[
    \frac{f(x)-f(q^2 x)}{x-q^2 x}
    \]
    is a continuous function on $sp(B_0),$ i.e.
    $$
    \lim_{k\to \infty}\frac{f(q^{2k})-f(q^{2k+2})}{q^{2k}-q^{2k+2}} \text{ exists.}
    $$
In this case we have  \begin{equation}
    (\db f)(q^{2k})=\frac{f(q^{2k})-f(q^{2k+2})}{q^{2k}-q^{2k+2}} \text{ and }(\db f)(0)= \lim_{k\to \infty}\frac{f(q^{2k})-f(q^{2k+2})}{q^{2k}-q^{2k+2}}.
\end{equation}
\end{proposition}
\begin{proof}
    Recall that the domain of $\d$ consists of all  functions $f\in C(sp(B_0))$ such that there exists a sequence $f_n\in \C[x]$ such that  $\lim_{n \to \infty} f_n = f$ and $\lim_{n \to \infty} \d f_n$ converges.
    
    Now, if $f\in \Dom (\db)$, let $f_n\in \C[x]$ be a sequence such that $f_n\to f$ and $\db f_n$ converges with limit $\db f$.  
    We know that 
    $$(\db f_n)(x)=\frac{f_n(x)-f_n(q^2 x)}{x-q^2 x}.$$
    Since $f_n\to f$ we get \[
    \frac{f(x)-f(q^2 x)}{x-q^2 x} = \lim_{n \to \infty} \frac{f_n(x)-f_n(q^2 x)}{x-q^2 x}.
    \]
    for any $x\neq 0$. Since $\db f_n$ converges we know $\frac{f(x)-f(q^2 x)}{x-q^2 x}$ is continuous on $sp(B_0)$ and 
    $$
    (\db f)(x)=\frac{f(x)-f(q^2 x)}{x-q^2 x}.
    $$

    On the other hand, if $\frac{f(x)-f(q^2 x)}{x-q^2 x}$ is a continuous function on $sp(B_0)$. Since $\C[x]$ is dense in $C(sp(B_0))$, there exists a sequence $g_n\in \C[x]$ such that
    $$g_n\to \frac{f(x)-f(q^2 x)}{x-q^2 x}$$
    Since $\I$ is a bounded operator on $C(sp(B_0))$ we get
    $$
   \I g_n\to\I \big(\frac{f(x)-f(q^2 x)}{x-q^2 x}\big).
    $$
    By \eqref{eq: expansion of I} we can check
    $$
    \I \big(\frac{f(x)-f(q^2 x)}{x-q^2 x}\big)=f(x)-f(0)
    $$
    therefore $\I g_n\to f-f(0)$. We then define
    $$f_n=\I g_n+f(0).
    $$
    It is then clear that $f_n\to f$ and $\db f_n=\db \I g_n=g_n$ converges with limit $\db f$.
\end{proof}

\begin{corollary}\label{coro: inverse domain d}
    If  $f\in C(sp(B_0))$ is in the domain of $\d$ and $f$ is invertible, then $f^{-1}$ is also in the domain of $\d$.
\end{corollary}
\begin{proof}
From the definition, we know that $f\in C(sp(B_0))$ is invertible if and only if $f(q^{2n})\neq 0$ for each $n\geq 0$ and $f(0)=\lim_{n\to \infty}f(q^{2n})\neq 0$. If we further know that $f$ is in the domain of $\d$, then by Proposition \ref{prop: delta_bar_domain} we have
  $$
    \lim_{k\to \infty}\frac{f(q^{2k})-f(q^{2k+2})}{q^{2k}-q^{2k+2}}= \text{ exists.}
    $$
    Therefore
      \begin{equation*}
      \begin{split}
    &\lim_{k\to \infty}\frac{f^{-1}(q^{2k})-f^{-1}(q^{2k+2})}{q^{2k}-q^{2k+2}} \\
    =& -\lim_{k\to \infty}\frac{1}{f(q^{2k})f(q^{2k+2})}\lim_{k\to \infty}\frac{f(q^{2k})-f(q^{2k+2})}{q^{2k}-q^{2k+2}}\text{ also exists.}
    \end{split}
    \end{equation*}
    Again by Proposition \ref{prop: delta_bar_domain} we know that $f^{-1}$ is in the domain of $\d$.
\end{proof}

\begin{corollary} \label{coro: db and pb correspond general}
    For any $f\in \Dom(\pb)\cap C^*(1, B_0)$ we have $\Psi(f)\in \Dom(\db)$. Moreover we have 
    \begin{equation}
        \pb f=\big (\Psi^{-1}(\db (\Psi f))\big )\pb B_0.
    \end{equation}
    Sometimes we abuse the notation and simply write it as 
        \begin{equation}\label{eq: db and pb correspond general}
        \pb f=(\db ( f))\pb B_0.
    \end{equation}
\end{corollary}

\subsection{Existence of non-standard holomorphic structures}\label{subsec: existence of non standard holo structures}
We can now tackle our problem of finding an invertible $g$ such that $\pb g = g \pb f,$ in the case of restricting both $g$ and $f$ to $C^*(1, B_0).$  

\begin{proposition} \label{prop: solution g in product form}
    Let $f\in C(sp(B_0))$ be a function contained in $\Dom(\db).$ Then, a solution to $\d g = g \d f$ is given by \begin{equation}\label{eq: def of g}
    g(q^{2n}) = g(1) \prod_{k = 1}^n  (1 - f(q^{2k - 2}) + f(q^{2k})) \text{ for any }n\geq 1,
    \end{equation}
    and \begin{equation}\label{eq: def of g0}
    g(0) = g(1)\prod_{k = 1}^\infty (1 - f(q^{2k - 2}) + f(q^{2k}))
    \end{equation}
   The $g$ defined above is always in $\Dom(\db)$.
\end{proposition}
\begin{proof}
    First, recall from Proposition \ref{prop: delta_bar_domain} that $f \in Dom(\db)$ if and only if $\frac{f(x) - f(q^2x)}{1 - q^2x}$ is a continuous function on $sp(B_0),$ and if so, $\d f = \frac{f(x) - f(q^2x)}{1 - q^2x}.$ Then, given such an $f,$ a solution $g$ such that $\d g = g \d f$ is equivalent to a $g$ such that 
    \begin{equation}\label{eq: solve for g 1}
    \frac{g(x) - g(q^2x)}{1 - q^2x} = g(x) \frac{f(x) - f(q^2x)}{1 - q^2x}
    \end{equation}
    Now, since we are only considering $x \in [0,1],$ $1 - q^2x$ is never zero, so we may reduce \eqref{eq: solve for g 1} to 
    $$
    g(x) - g(q^2x) = g(x) (f(x) - f(q^2x)),
     $$
     which gives
    \begin{equation}\label{eq: solve for g 2}
    g(q^2x) = g(x)(1 - f(x) + f(q^2x)).
      \end{equation}
    Now, if we let $x = q^{2n - 2},$ this gives us the recursive formula \[
    g(q^{2n}) = g(q^{2n - 2})(1 - f(q^{2n - 2}) + f(q^{2n})).
    \]
    Thus, if we write $g(1) = c$ for any $c \in \C,$ we obtain \begin{equation}\label{eq: def of g in proof}
    g(q^{2n}) = c \prod_{k = 1}^n (1 - f(q^{2n - 2k}) + f(q^{2n - 2k + 2})) = c \prod_{k = 1}^n (1 - f(q^{2k - 2}) + f(q^{2k}))
    \end{equation}
    
We know $g$ is continuous if and only if $\lim_{n \to \infty} g(q^{2n}) = g(0).$ Since the $g$ is defined pointwise in \eqref{eq: def of g in proof}, we need only set $g(0)$ as the limit to the above expression as $n \to \infty.$ Then, $g$ is continuous if the product \[
    g(0) := c \prod_{k = 1}^\infty (1 - f(q^{2k - 2}) + f(q^{2k}))
    \]
   converges. 
   By taking logarithm, it is easy to see that the above infinite product converges if \[
    \sum_{j=1}^\infty\big( f(q^{2n})- f(q^{2n-2})\big)
    \]
    converges absolutely. 
     Since $f\in \Dom(\db),$ we have \[
    |f(q^{2n})- f(q^{2n-2})| = |(q^{2n} - q^{2n - 2}) \d f(q^{2n - 2}) |
    \]
where $\db f$ is a continuous function on $sp(B_0)$, hence bounded. As a result there exists a number $M$ such that
$$|f(q^{2n})- f(q^{2n-2})|\leq M(1-q^2)q^{2n-2}$$
for all $n$. Hence $ \sum_{j=1}^\infty\big( f(q^{2n})- f(q^{2n-2})\big)$ converges absolutely hence $g$ is continuous.

Now $g$ is continuous and satisfies \eqref{eq: solve for g 1} for any $x=q^{2n}$. By Proposition \ref{prop: delta_bar_domain} the right hand side of \eqref{eq: solve for g 1} is continuous, hence the left hand side, which is  $\frac{g(x) - g(q^2x)}{1 - q^2x}$, is also continuous. Again by Proposition \ref{prop: delta_bar_domain} we know that $g\in \Dom(\db)$.
\end{proof}

\begin{corollary}\label{coro: g invertible}
   The solution $g$ in Proposition \ref{prop: solution g in product form} is invertible if and only if $g(1)\neq 0$ and
    \begin{equation}\label{eq: defective spot of f}
    f(q^{2n}) - f(q^{2n - 2}) \neq 1\text{ for all }n \in \N.\end{equation}
\end{corollary}
\begin{proof}
We know that $g$ is invertible if and only if $g(q^{2n})\neq 0$ for each $n\geq 0$ and $g(0)=\lim_{n\to \infty}g(q^{2n})\neq 0$. If $g(1)=0$ then $g$ is clearly non-invertible. So now we assume $g(1)\neq 0$.
    
    By \eqref{eq: def of g}, we know that $g(q^{2n}) = 0$ for some $n$ if and only if there exists some $k \leq n$ such that $f(q^{2k}) - f(q^{2k - 2}) = 1.$ 
    
    Also, if $ (1 - f(q^{2k - 2}) + f(q^{2k})) \neq 0$ for each $k$, then since 
    $$\sum_{k=1}^{\infty}\log((1 - f(q^{2k - 2}) + f(q^{2k})))
    $$ does not go to $-\infty$ as in the proof of Proposition \ref{prop: solution g in product form}, we know that the infinite product
    $$g(0)= g(1)\prod_{k = 1}^\infty (1 - f(q^{2k - 2}) + f(q^{2k}))$$ is also  not zero. We finished the proof.
\end{proof}

Inspired by Corollary \ref{coro: g invertible} we have the following definition.

\begin{definition}\label{def: defective spot}
    We say that $f \in C^*(1, B_0)$ has a \textit{defective spot} at $n \in \N$ if $$(\Psi f)(q^{2n}) - (\Psi f)(q^{2n - 2}) = 1,$$
    where $\Psi: C^*(1, B_0)\overset{\sim}{\to}C(sp(B_0))$ is the functional calculus isomorphism as in \eqref{eq: functional calculus B0}. 

    We denote the set of defective spots of $f$ by $S_f$.
\end{definition}
\begin{remark}\label{rmk: defective spot finite}
We know that $S_f$ must be a finite subset of $\N$ as $\Psi f\in C(sp(B_0))$ is continuous at $0$.
\end{remark}

Note that $\frac{B_0}{q^{2n-2} - q^{2n}}$ is a function which has a defective spot at $n.$

\begin{theorem}\label{thm: gauge equivalent to trivial criterion}
    Given $f \in  \A(\C P^1_q)\cap C^*(1, B_0),$ there exists an invertible $g\in\mathcal{C}(\C P^1_q)^{\times}\cap \Dom(\pb)$ such that $g^{-1}\in \Dom(\pb)$ and $\pb g = g \pb f$ if and only if $f$ has no defective spot.

    In other words, for $f \in  \A(\C P^1_q)\cap C^*(1, B_0),$ the $\pb$-connection $\nb$ on $\L_n$ with $\theta=\pb f$ is gauge equivalent to the standard $\pb$-connection $\overline{\nabla}$ if and only if $f$ has no defective spot.
\end{theorem}
\begin{proof}
    By Corollary \ref {coro: db and pb correspond general} and Corollary \ref{coro: g invertible}, if $f$ has no defective spot, then an invertible solution $g \in\mathcal{C}(\C P^1_q)^{\times}\cap  C^*(1, B_0)\cap \Dom(\pb)$ to $\pb g = g \pb f$ exists. By Corollary \ref{coro: inverse domain d} and  Corollary \ref {coro: db and pb correspond general}, we know that $g^{-1}\in \Dom(\pb)$.
    
    On the other hand, if $f$ has a defective spot, then by Corollary \ref{coro: g invertible}, $\pb g = g \pb f$  has a  non-invertible, nonzero solution. By Lemma \ref{lemma: non zero solution non gauge equivalent}, $\pb g = g \pb f$ cannot have any invertible solution with $g$ and $g^{-1}$ in $\Dom(\pb)$.
\end{proof}

\begin{example}\label{example: non gauge equivalent to zero}
We notice that the element $B_0$ has no defect spot. Actually in Example \ref{example: exp equation for B_0} we found explicitly an invertible element $g$ such that $g^{-1}\in \Dom(\pb)$  and $\pb g=g \pb B_0$.

    On the other hand we consider $f=\frac{B_0}{1-q^2}\in \A(\C P^1_q)$. It is clear that the defective spot $S_f=\{1\}$. By \eqref{eq: pb of B_0^n} we can get
    \begin{equation}
    \pb(B_0^{\infty})=B_0^{\infty}\pb(f)
    \end{equation}
    where 
    $$
    B_0^{\infty}=\lim_{n\to\infty}B_0^n\in C^*(1, B_0).
    $$
    Since $\Psi(B_0)=x$ we have
    $$
    \Psi(B_0^{\infty})(q^{2n})=\begin{cases}1 &n=0\\0 & n\geq 1
    \end{cases}
    $$
    which is a continuous function on $sp(B_0)=\{0\} \cup \{q^{2n} | n \in \Z_{\geq 0} \}$.
    In particular $B_0^{\infty}$ is not invertible. Therefore $\overline{\nabla}_{\pb(\frac{B_0}{1-q^2})}$ is not gauge equivalent to the standard $\pb$-connection $\overline{\nabla}$, which gives a concrete example of non-standard holomorphic structure on $\L_n$.
\end{example}

On the other hand, we have the following affirmative result for $\pb$-connections which are gauge equivalent to the standard ones.

\begin{corollary}\label{coro: less than one half gauge equivalent to zero}
    For any $f\in  \A(\C P^1_q)\cap C^*(1, B_0)$ with $\|f\|<\frac{1}{2}$, the $\pb$-connection $\overline{\nabla}_{\pb f}$ on $\L_n$ is gauge equivalent to the standard $\pb$-connection $\overline{\nabla}$.
\end{corollary}
\begin{proof}
    Since $\|f\|<\frac{1}{2}$, we know that 
    $$|f(q^{2k})-f(q^{2k-2})|<1 \text{ for any }k,$$ hence $f$ cannot have defective spot.
\end{proof}

\subsection{Gauge equivalence between $\pb$-connections}\label{subsec: gauge equivalence C(1, B0)}
We now turn to the question that when $\overline{\nabla}_{\pb f}$ and $\overline{\nabla}_{\pb h}$ are gauge equivalent for $f, h \in \A(\C P^1_q)\cap C^*(1, B_0)$.

This means that the existence of a non-invertible $g$ such that $\pb g = g \pb f + \pb h \cdot g$ does not mean that $f$ and $h$ must lie in different gauge equivalency classes. However, the existence of such an invertible $g$ still implies that $f$ and $h$ are gauge equivalent.

\begin{lemma} \label{lemma: commuting partials in CB0}
    For $g, h \in \Dom(\pb)\cap C^*(1, B_0)$ we have  
    \begin{equation}\label{eq: commuting partials in CB0}
    \pb h \cdot g = (m_{q^2}g) \cdot \pb h,
    \end{equation}
    where $m_{q^2}$ is the dilation map in \eqref{eq: mc} extended to $C^*(1, B_0)$ via the functional calculus isomorphism.
\end{lemma}
\begin{proof}
   By \eqref{eq: pb of B_0^n} we get
   \begin{equation}\label{eq: commuting partial B0 1}
       \pb(B_0^{m+n})=\frac{1-q^{2m+2n}}{1-q^2}B_0^{m+n-1}\pb B_0.
   \end{equation}
   On the other hand we have
     \begin{equation}\label{eq: commuting partial B0 2}\pb(B_0^{m+n})=\pb(B_0^m)B_0^n+B_0^m\pb(B_0^n)\end{equation}
   where 
   $$
   B_0^m\pb(B_0^n)=B_0^m\frac{1-q^{2n}}{1-q^2}B_0^{n-1}\pb B_0=\frac{1-q^{2n}}{1-q^{2m+2n}}\pb(B_0^{m+n}).
   $$
   Therefore \eqref{eq: commuting partial B0 2} becomes
    $$ 
      \frac{1-q^{2m+2n}}{1-q^{2n}}B_0^m\pb(B_0^n)=\pb(B_0^m)B_0^n+B_0^m\pb(B_0^n)
 $$  
   hence
       \begin{equation}\label{eq: commuting partial B0 3}
    \pb(B_0^m)B_0^n=q^{2m}B_0^m\pb(B_0^n)=(m_{q^2}B_0^m)\pb(B_0^n).
   \end{equation}
   We proved that \eqref{eq: commuting partials in CB0} holds for any monomials hence for any polynomials. The general case now follows from the fact that $\pb$ is a closed operator, together with the fact that multiplication and $m_{q^2}$ are bounded operators.
\end{proof}

\begin{proposition}\label{prop: gauge equivalent between connections}
    Let $f, h \in \A(\C P^1_q)\cap C^*(1, B_0),$ and write $S_f, S_h \subset \N$ for the sets of defective spots of $f, h$ respectively. Then, there exists an invertible $g \in C^*(1, B_0)^{\times}\cap \Dom(\db)$ such that \begin{equation}\label{eq: gauge equivalence in B0}
    \pb g = g \pb f -\pb h \cdot g\end{equation} if and only if $S_f = S_h.$

    In particular if $S_f = S_h$, then the two $\pb$-connections $\overline{\nabla}_{\pb f}$ and $\overline{\nabla}_{\pb h}$  are gauge equivalent.
\end{proposition}
\begin{proof}
The second assertion follows from the first one together with  Definition \ref{def: gauge equivalence} and Corollary \ref{coro: inverse domain d}.

    To prove the first assertion, we again use the functional calculus isomorphism $\Psi$ to identify $C^*(1, B_0)$ and $C(sp(B_0))$. Equation \eqref{eq: gauge equivalence in B0} then becomes $$\d g = g \d f - (\d h) g.$$
    By Lemma \ref{lemma: commuting partials in CB0} it becomes 
    \begin{equation}\label{eq: gauge equivalent dilated}
    \d g = g \d f - (m_{q^2} g) \d h,
    \end{equation}
    By Proposition \ref{prop: delta_bar_domain}  we can write \eqref{eq: gauge equivalent dilated}  as \begin{equation}
    \frac{g(x) - g(q^2x)}{x - q^2x} = g(x) \frac{f(x) - f(q^2x)}{x - q^2x} - g(q^2x) \frac{h(x) - h(q^2x)}{x - q^2x}
    \end{equation}
    Since $x - q^2x$ is never zero for $x \in (0,1],$ this becomes \[
    g(x) - g(q^2x) = g(x) (f(x) - f(q^2x)) - g(q^2x) (h(x) - h(q^2x))
    \]
    hence \begin{equation}\label{eq: recursive eq for g}
    g(q^2x)[1 - h(x) + h(q^2x)] = g(x)[1 - f(x) + f(q^2x)]
    \end{equation}
    For $x=q^{2n}$, \eqref{eq: recursive eq for g} becomes
    \begin{equation}\label{eq: recursive eq for g 2}
    g(q^{2n+2})[1 - h(q^{2n}) + h(q^{2n+2})] = g(q^{2n})[1 - f(q^{2n}) + f(q^{2n+2})]
    \end{equation}

If $S_f\neq S_h$, then there must exist an $n$ such that one of $1 - f(q^{2n}) + f(q^{2n+2})$ and $1 - h(q^{2n}) + h(q^{2n+2})$ is zero and the other is non-zero, hence one of $g(q^{2n})$ and $g(q^{2n+2})$ must be zero. Therefore $g$ cannot be invertible.

    If $S_f= S_h$, then $1 - f(q^{2n}) + f(q^{2n+2})$ and $1 - h(q^{2n}) + h(q^{2n+2})$ are both zero or both nonzero. If both are nonzero, then we have 
    \begin{equation}
    g(q^{2n+2})=g(q^{2n})\frac{1 - f(q^{2n}) + f(q^{2n+2})}{1 - h(q^{2n}) + h(q^{2n+2})}.
    \end{equation}
    If both are zero, then \eqref{eq: recursive eq for g 2} implies that $g(q^{2n+2})$ can be any number. We can therefore define $g$ inductively at any $q^{2n}$ so that $g(q^{2n})\neq 0$.

   Moreover since $S_f=S_h$ is a finite set, let 
   $$
   N= \text{the maximum of }S_f.
   $$
Then for any $n>N$, $1 - f(q^{2n}) + f(q^{2n+2})$ and $1 - h(q^{2n}) + h(q^{2n+2})$ are both nonzero hence $g(q^{2n})$ is uniquely determined by $g(q^{2N})$ by
\begin{equation}
    g(q^{2n}) = g(q^{2N})\prod_{k=N}^{n-1} \frac{1 - f(q^{2k}) + f(q^{2k+2})}{1 - h(q^{2k}) + h(q^{2k+2})} .
    \end{equation}

    Therefore
  \begin{equation}\label{eq: infinity product quotient}
    \lim_{n \to \infty} g(q^{2n}) 
  =\prod_{k = N }^{\infty}  \frac{1 - f(q^{2k}) + f(q^{2k + 2})}{1 - h(q^{2k}) + h(q^{2k + 2})}
    \end{equation}

   Since $f,h\in \Dom(\db)$, by the same argument as in the proof of Proposition \ref{prop: solution g in product form} and Corollary \ref{coro: g invertible}, the infinite product on the right hand side of \eqref{eq: infinity product quotient} converges with a nonzero limit. Thus, $g(0)$ exists, and is nonzero.
   Hence $g$ is a continuous function which is invertible.
    
    It remains to show that $g\in \Dom(\db)$. But this  follows from \[
    \d g = g \d f - \d h \cdot g,
    \]
    and the fact that $\d f, \d h,$ and $g$ are all continuous.
\end{proof}

\begin{remark}
    Notice that Proposition \ref{prop: gauge equivalent between connections} gives a sufficient but not necessary condition: if $S_f\neq S_h$, we do not know if $\overline{\nabla}_{\pb f}$ and $\overline{\nabla}_{\pb h}$ are gauge equivalent or not. The main reason is that we do not have a generalization of Lemma \ref{lemma: non zero solution non gauge equivalent} to solutions of
    $$
    \pb g=g \pb f-\pb h g.
    $$

    We will study non-gauge equivalent $\pb$-connections using a different method in Section \ref{section: holomorphic sections}.
\end{remark}

\section{Holomorphic Sections of Non-Standard Line bundles}\label{section: holomorphic sections}

By Corollary \ref{coro: any connection gives a spectral triple}, for any $\pb$-connection $\nb$ on $\L_0$ the space of holomorphic sections
$
\ker (\nb)
$ is finite dimensional. In this section we look for elements in $\ker(\overline{\nabla}_\theta)\subset \L_0=\A(\C P^1_q)$ of the form $fB_-^n $ for some $n \in \Z_{\geq 0}$,  where $f\in C^*(1, B_0)$.

We first prove the following results:

\begin{lemma}\label{lemma: commuting partials in CB0 and B-}
For any $f\in C^*(1,B_0)$ and $n\in \N$ we have
\begin{equation}\label{eq: commuting partials in CB0 and B-}
 B_-^n ~f=(m_{q^{2n}}f)B_-^n,
\end{equation}
where $m_{q^{2n}}$ represents a dilation operator as in \eqref{eq: mc}.
\end{lemma}
\begin{proof}
    Similar to the proof of Lemma \ref{lemma: commuting partials in CB0}, by  \eqref{eq: B0 pb B- B+ commutes}  we can check that \eqref{eq: commuting partials in CB0 and B-} holds when $f$ is a polynomial. The general case follows by  continuity of multiplications and $m_{q^{2n}}$.
\end{proof}

\begin{lemma}\label{lemma: nabla of f B-}
    For  $f\in \Dom(\pb)\cap C^*(1, B_0)$, $h\in  \A(\C P^1_q)\cap C^*(1, B_0)$, and $\theta=\pb h$, we have
        \begin{equation}
            \nb (f B_-^n) = \Big( q^{2n} B_0 \db f + \big(\frac{1-q^{2n}}{1-q^2} - q^{2n} (m_{q^{2n}} \db h) B_0\big) f\Big) B_-^{n-1} \pb B_-,
        \end{equation}
     Here we abuse the notation and denote $\Psi^{-1}(\db (\Psi f))$ simply by $\db f$.
\end{lemma}
\begin{proof}
    By the definition of $\nb$   we get $$
    \nb(f B_-^n) = \pb(f B_-^n) - (f B_-^n)\theta= \pb(f B_-^n) - (f B_-^n)\pb h.
    $$
Notice that we are working with $\L_0$ hence there is no need of $\Phi_{(n)}$ as in \eqref{eq: pb-connections omega}. 

By  \eqref{eq: B0 pb B- B+ commutes}  and \eqref{eq: db and pb correspond general} we further get
    \begin{equation}\label{eq: nb f B- 1}
    \begin{split}
   \nb(f B_-^n) & = (\d f) \pb B_0 B_-^n + f \pb(B_-^n) - f B_-^n \pb h\\
    =& q^2 (\d f) B_0 (\pb B_-) B_-^{n-1} + \frac{1-q^{2n}}{1-q^2} f B_-^{n-1} \pb B_- - f B_-^n \pb h\\
    =& q^{2n} (\d f) B_0 B_-^{n-1} \pb B_- + \frac{1-q^{2n}}{1-q^2} f B_-^{n-1} \pb B_- - f B_-^n \pb h.
    \end{split}
    \end{equation}

    We write $\pb h = \d h \pb B_0.$ Then we have 
    \begin{equation}\label{eq: B-n pb h}
    B_-^n \pb h = B_-^n \d h \pb B_0 =  (m_{q^{2n}} \d h) B_-^n \pb B_0.
    \end{equation}
    By \eqref{eq: B0 pb B- B+ commutes}, $\pb B_0$ commutes with $B_-,$ so the right hand side of \eqref{eq: B-n pb h} becomes \begin{equation}\label{eq: mq db h}
    (m_{q^{2n}} \d h) (\pb B_0) B_-^n = q^2 (m_{q^{2n}} \d h) B_0 (\pb B_-) B_-^{n-1} = q^{2n} (m_{q^{2n}} \d h) B_0 B_-^{n-1} \pb B_-,
    \end{equation}
    \eqref{eq: mq db h} together with \eqref{eq: nb f B- 1} give
    \begin{equation}\label{eq: nb f B- 2} 
    \begin{split}
    \nb(f B_-^n) =& q^{2n} (\d f) B_0 B_-^{n-1} \pb B_- +  \frac{1-q^{2n}}{1-q^2} f B_-^{n-1} \pb B_- - q^{2n} f (m_{q^{2n}} \d h) B_0 B_-^{n-1} \pb B_-\\
    =& \Big( q^{2n} (\d f) B_0 +  \frac{1-q^{2n}}{1-q^2}f  - q^{2n} f (m_{q^{2n}} \d h) B_0 \Big) B_-^{n-1} \pb B_- 
    \end{split}
    \end{equation}

    Since everything in these large parentheses is in the commutative $C^*$-algebra $C^*(1, B_0),$ we can rewrite \eqref{eq: nb f B- 2} as \[
    \nb(f B_-^n)=\Big( q^{2n} B_0 \d f + \big( \frac{1-q^{2n}}{1-q^2}  - q^{2n} (m_{q^{2n}} \d h) B_0\big) f \Big) B_-^{n-1} \pb B_- 
    \]  
\end{proof}

\begin{corollary}\label{coro: ker nb B-n}
   Let $h\in  \A(\C P^1_q)\cap C^*(1, B_0)$ and  $\theta = \d h$. Consider the $\pb$-connection $\nb$ on $\L_0$. Suppose the defective spot $S_h\neq \emptyset$. Then for any $0\leq n< \max{S_h}$, there exists an element $f\in \Dom(\pb)\cap   C^*(1, B_0)$ such that 
   $$
   (\Psi f)(1)\neq 0, \text{ and } f B_-^n\in \ker(\nb),
   $$
   where $\Psi: C^*(1, B_0)\overset{\sim}{\to}C(sp(B_0))$ is the functional calculus isomorphism
\end{corollary}
\begin{proof}
Again we use $\Psi$ to identify $C^*(1, B_0)$ and $ C(sp(B_0))$. By Lemma \ref{lemma: nabla of f B-}, to find $f$ such that $f B_-^n\in \ker(\nb)$, it is sufficient to find an $f\in \Dom(\db)\cap   C(sp(B_0))$ such that
\begin{equation}
  q^{2n} B_0 \d f + \big( \frac{1-q^{2n}}{1-q^2}  - q^{2n} (m_{q^{2n}} \d h) B_0\big) f=0,
\end{equation}
i.e. for any $q^{2k}\in sp(B_0)$, $k\geq 0$ we have 
\begin{equation}\label{eq: nb f B- 1}
     q^{2k+2n}  \frac{f(q^{2k}) - f(q^{2k+2})}{q^{2k} - q^{2k+2}} + \frac{1-q^{2n}}{1-q^2}f(q^{2k})  - q^{2n} \frac{h(q^{2k+2n}) - h(q^{2k+2n+2})}{q^{2k+2n} - q^{2k+2n+2}} q^{2k} f(q^{2k})  = 0
\end{equation}

    From \eqref{eq: nb f B- 1} we get
\begin{equation}
        f(q^{2k+2})=\frac{1-h(q^{2k+2n})+h(q^{2k+2n+2})}{q^{2n}}f(q^{2k}).
\end{equation}
Therefore for any $m\geq 1$ we have
\begin{equation}\label{eq: nb fB- recursive}
f(q^{2m})=  f(1) \prod_{k = 1}^{m} \frac{1 - h(q^{2n + 2k-2}) + h(q^{2n+2k })}{q^{2n}} .
\end{equation}

Since $n<\max{S_h}$, there exists $m_0> 0$ such that $n + m_0 \in S_h$. By \eqref{eq: nb fB- recursive} we have
$$
f(q^{2m})=0, \text{ for any }m \geq m_0.
$$
Therefore we can choose $f(1)\neq 0$ and the function $f$ defined by \eqref{eq: nb fB- recursive} is continuous and belongs to $\Dom(\db)$. Moreover it satisfies $\nb(fB_-^n)=0$.
\end{proof}

\begin{remark}
    If $n\geq \max{S_n}$, then 
    $$1 - h(q^{2n + 2k-2}) + h(q^{2n+2k })\neq 0 \text{ for any } k>0,$$  and 
    $$\lim_{k\to \infty}\big(1 - h(q^{2n + 2k-2}) + h(q^{2n+2k })\big)=1.$$
    Since $0<q^{2n}<1$, the $f(q^{2m})$ defined by \eqref{eq: nb fB- recursive} diverges unless $f(1)=0$.
\end{remark}

\begin{lemma}\label{lemma: fB- linearly independent}
Let $n_1,\ldots, n_k$ be distinct nonnegative integers. Then for any $f_1,\ldots f_k\in C^*(1,B_0)$ such that $(\Psi f_i)(1)\neq 0$ for each $i$, the elements $f_1B_-^{n_1},\ldots, f_kB_-^{n_k}$ are linearly independent over $\C$.
\end{lemma}
\begin{proof}
    Suppose we have $c_1,\ldots, c_k\in \C$ which are not all zeros.
    Let $n_s$ be the smallest $n_i$ such that $c_i\neq 0$. Recall the faithfull representation $\pi: \mathcal{C}(\C P^1_q)\to B(\ell^2)$ in Proposition \ref{prop: standard representation}. It is clear that$(\Psi f_s)(1)\neq 0$ implies
    $\pi(f_s)(e_0)\neq 0$. Moreover \eqref{eq: standar representation} implies
    \begin{equation}
        \pi(B_-^{n_s})\pi(B_+^{n_s})(e_0)=\lambda e_0, \text{ for some }\lambda\neq 0,
    \end{equation}
    and
    \begin{equation}
          \pi(B_-^{n})\pi(B_+^{n_s})(e_0)=0, \text{ for any }n> n_s.
    \end{equation}

    We apply $\pi(\sum_{i=1}^kc_if_iB_-^{n_i})$ to the vector $\pi(B_+^{n_s})(e_0)\in \ell^2$ and get
    \begin{equation}
    \begin{split}
    &\pi(\sum_{i=1}^kc_if_iB_-^{n_i})\pi(B_+^{n_s})(e_0)=\\
    =& c_s\lambda (\pi (f_s))(e_0)+\sum_{n_i>n_s}c_i(\pi (f_i)) \pi(B_-^{n_i})\pi(B_+^{n_s})(e_0)\\
    =& c_s \lambda (\pi (f_s))(e_0)+0=c_s \lambda (\pi (f_s))(e_0)\neq 0.
    \end{split}
    \end{equation}
    So we have $\sum_{i=1}^kc_if_iB_-^{n_i}\neq 0$.
\end{proof}

\begin{corollary} \label{coro: ker_dim_lower_bound}
    Let $h\in  \A(\C P^1_q)\cap C^*(1, B_0)$ and $N$ be the maximal element in $S_h$. Then for  $\theta = \pb h$ and the $\pb$-connection $\nb$ on $\L_0$, we have 
    \begin{equation}
    \dim(\ker(\nb)) \geq N.\end{equation}
\end{corollary}
\begin{proof}
    It is a direct consequence of Corollary \ref{coro: ker nb B-n} and Lemma \ref{lemma: fB- linearly independent}.
\end{proof}

\begin{remark}
If we want to extend the result to $B_+$, then we notice that we have an analogue of Lemma \ref{lemma: fB- linearly independent} for  $B_+^{n_1}f_1,\ldots, B_+^{n_k}f_k$ instead of $f_1B_+^{n_1},\ldots, f_kB_+^{n_k}$.

However, a careful computation shows
\begin{equation}
      \nb(B_+^n f) = B_+^{n-1}  \pb B_+ \Big( \frac{1-q^{2n}}{1-q^2} f +  (q^{-4} B_0 - q^{-2})
    \big((m_{q^{-2}} \db f) - (m_{q^{-2}} f)(m_{q^{-2}} \db h)\big)  \Big).
\end{equation}
Since $q^{-2} > 1,$ this dilation operator $m_{q^{-2}}$ is unbounded, so we cannot use functional calculus to solve this equation.
\end{remark}

The following theorem is the main result of this paper:

\begin{theorem}\label{thm: infinitely many gauge class}
    There exist infinitely many gauge equivalent classes of holomorphic structures on $\L_0$, hence on $\L_n$.
\end{theorem}
\begin{proof}
    We know  that the element $$
    h=\frac{B_0}{q^{2N-2}-q^{2N}}
    $$ has  $S_h=\{N\}$. Therefore by Corollary \ref{coro: ker_dim_lower_bound}, for any $N$, we can find a $\pb$-connection $\nb$ on $\L_0$ such that 
    $$
    \dim(\ker(\nb)) \geq N.
    $$

On the other hand, by Corollary \ref{coro: any connection gives a spectral triple}, $\ker(\nb)$ is finite dimensional for any  $\pb$-connection $\nb$ on $\L_0$. Since the dimension of $\ker(\nb)$  is invariant under gauge equivalence, there exist infinitely many gauge equivalent classes of holomorphic structures on $\L_0$.

The $\L_n$ case follows from the $\L_0$ case and Proposition \ref{prop: one-one corrspondence gauge equivalence class}.
\end{proof}

\begin{remark}
It is a classical result that on commutative $\C P^1$, there exists a unique holomorphic structure up to gauge equivalence on each $\mathcal{O}(n)$. Therefore the existence of infinitely many holomorphic structures in Theorem \ref{thm: infinitely many gauge class} is a new phenomenon in noncommutative geometry which has no counterpart in the commutative world.
\end{remark}

\section{Future Work}\label{section: future work}

Note that Theorem \ref{thm: infinitely many gauge class} does not provide a classification of the gauge equivalence classes of holomorphic structures on $\L_n$ over $\C P^1_q$. It would be interesting to classify and parametrize all such gauge equivalence classes, that is, to determine the Picard group of the quantum projective line $\C P^1_q$.

We also notice that higher dimensional quantum projective spaces $\C P^l_q$ and line bundles over them were  studied in \cite{d2008noncommutative}, \cite{d2010anti}, \cite{d2010dirac}, \cite{d2013geometry}, and \cite{khalkhali2011noncommutative}. It is interesting to study non-standard holomorphic structures on line bundles over $\C P^l_q$. The analysis will be more involved in higher dimensional case as the flatness condition $\nb\circ \nb=0$ does not hold automatically on $\C P^l_q$  for $l\geq 2$.

\bibliographystyle{plain}
\bibliography{refs}

\end{document}